\documentclass{article}
\usepackage[margin=2cm]{geometry}                % See geometry.pdf to learn the layout options. There are lots.
\usepackage{amsmath,amssymb,trfsigns,amsthm}
\usepackage{graphics}
\usepackage{amssymb}
\usepackage{hyperref}
\usepackage{xcolor}
\usepackage[utf8]{inputenc} 
\usepackage{multirow}
\usepackage{textcomp}
\usepackage{mathtools}
\usepackage{epsfig} 

\newtheorem{theorem}{Theorem}[section]
\newtheorem{assumption}{Assumption}[section]

\definecolor{darkgreen}{rgb}{0, 0.5, 0}
\definecolor{darkred}{rgb}{0.5, 0, 0}

\title{\LARGE \bf
A simulation-based approach for solving optimisation problems with ODE-type steady state constraints \thanks{The authors acknowledge financial support from the Postdoctoral Fellowship Program (PFP) of the Helmholtz Zentrum M\"unchen and the German Federal Ministry of Education and Research (BMBF, Grant no. 01ZX1310B) within the SYS-Stomach project.}
}

\author{Anna Fiedler\thanks{Institute of Computational Biology, Helmholtz Zentrum M\"unchen, Ingolst\"adter Landstra{\ss}e 1, 85764 Neuherberg, Germany; and Chair of Mathematical Modeling of Biological Systems, Center for Mathematics, Technische Universit\"at M\"unchen, Boltzmannstra{\ss}e 3, 85748~Garching, Germany.
        {\tt\small anna.fiedler@tum.de, fabian.theis@helmholtz-muenchen.de, jan.hasenauer@helmholtz-muenchen.de}} \and Fabian J. Theis\footnotemark[2] \and Jan Hasenauer\footnotemark[2] %
}
\date{}

\begin{document}

\maketitle
\thispagestyle{empty}
\pagestyle{empty}

%%%%%%%%%%%%%%%%%%%%%%%%%%%%%%%%%%%%%%%%%%%%%%%%%%%%%%%%%%%%%%%%%%%%%%%%%%%%%%%%
\begin{abstract}

Ordinary differential equations (ODEs) are widely used to model biological, (bio-)chemical and technical processes. The parameters of these ODEs are often estimated from experimental data using ODE-constrained optimisation. This article proposes a simple simulation-based approach for solving optimisation problems with steady state constraints relying on an ODE. This simulation-based optimisation method is tailored to the problem structure and exploits the local geometry of the steady state manifold and its stability properties. A parameterisation of the steady state manifold is not required. We prove local convergence of the method for locally strictly convex objective functions. Efficiency and reliability of the proposed method are demonstrated in two examples. The proposed method demonstrated better convergence properties than existing general purpose methods and a significantly higher number of converged starts per time.

\end{abstract}

%%%%%%%%%%%%%%%%%%%%%%%%%%%%%%%%%%%%%%%%%%%%%%%%%%%%%%%%%%%%%%%%%%%%%%%%%%%%%%%%
\section{Introduction}

Ordinary differential equation models are widely used in engineering, physics and life sciences. As the parameters of the underlying processes are often unknown, they have to be determined from experimental data~\cite{TarantolA2005}. This requires efficient and robust parameter optimisation methods. Unfortunately, for a wide range of models these efficiency and robustness requirements are not met by available methods (see discussion~\cite{RaueSch2013}). In particular parameter estimation and optimisation problems with nonlinear equality constraints, e.g., steady state constraints, are intricate.

In the following we will consider ODE models
\begin{equation}
\frac{dx}{dt} = f(\theta,x), \quad x(0) = x_0
\label{eq: model}
\end{equation}
with state variables $x(t) \in \mathbb{R}^{n_x}$ and parameters $\theta \in \mathbb{R}^{n_\theta}$. The vector field $f:\mathbb{R}^{n_\theta} \times \mathbb{R}^{n_x}  \to \mathbb{R}^{n_x}$ is assumed to be Lipschitz continuous to ensure existence and uniqueness of the solution. The parameter dependent steady states $x_s(\theta)$ of such systems are constrained by
\begin{equation}
0 = f(\theta,x_s).
\label{eq: ss}
\end{equation}
This defines a nonlinear manifold. As many systems operate in steady state and/or undergo fast equilibration, constrained optimisation problems
\begin{equation}
\begin{aligned}
\min_{x_s,\theta} &\; J(\theta,x_s)\\
\mathrm{s.t.} \, &\;  f(\theta,x_s) = 0,
\end{aligned}
\label{eq: con opt prob}
\end{equation}
with objective function $J:\mathbb{R}^{n_\theta} \times \mathbb{R}^{n_x}\to\mathbb{R}$, are widely studied. Such constrained optimisation problems arise, e.g., in systems biology when model parameters are estimated from steady state data and in engineering when the steady state of a process is designed.

%These optimisation problems are frequently addressed by local optimisation techniques which perform well~\cite{RaueSch2013}.
A variety of methods has been developed for optimisation under nonlinear equality constraints~\eqref{eq: con opt prob}. In particular local and manifold optimisation are used in practice~\cite{AbsilMah2007,Bertsekas1999}. Standard local optimisers have to move on the nonlinear manifold defined by the steady state constraint (Figure~\ref{fig: concept figure}). These methods become quickly inefficient and/or have convergence problems~\cite{Bertsekas1999}. Improvements are achieved by reformulating the constrained to an unconstrained optimisation problem using Lagrangian multipliers. However, this increases the problem size. Algorithms for manifold optimisation generally circumvent an increase of the problem size and moves on the manifold. Instead retraction operators are employed~\cite{AbsilMah2007}, projecting a point onto the manifold. Retraction operators which can be efficiently evaluated are unfortunately only available for a few classes of manifolds. For vector fields $f$ which are non-linear, there are generally no analytical functions enabling the projections, limiting the application of established algorithms for optimisation on manifolds.

In the following, we will propose a simulation-based framework for numerical optimisation with steady state constraints. Using ideas from continuous analogues of optimisers~\cite{DuerrEbe2011,Tanabe1979,Tanabe1985} and systems theory~\cite{KhalilBook2002}, we will introduce an ODE system exploiting the local geometry of the steady state manifold to move on the manifold. The manifold will be stabilised using a continuous retraction, available because we are interested in stable steady states of the model~\eqref{eq: model}. The simulation of the resulting ODE system with sophisticated adaptive step-size solvers yields an improved optimisation performance. We provide a proof for exponential stability of a locally optimal point for a locally strictly convex objective function $J(\theta,x)$ and a locally exponentially stable system \eqref{eq: model}. To evaluate the developed optimisation method, two application examples from the field of systems biology are studied. 

\begin{figure}[t]
\begin{centering}
\includegraphics[]{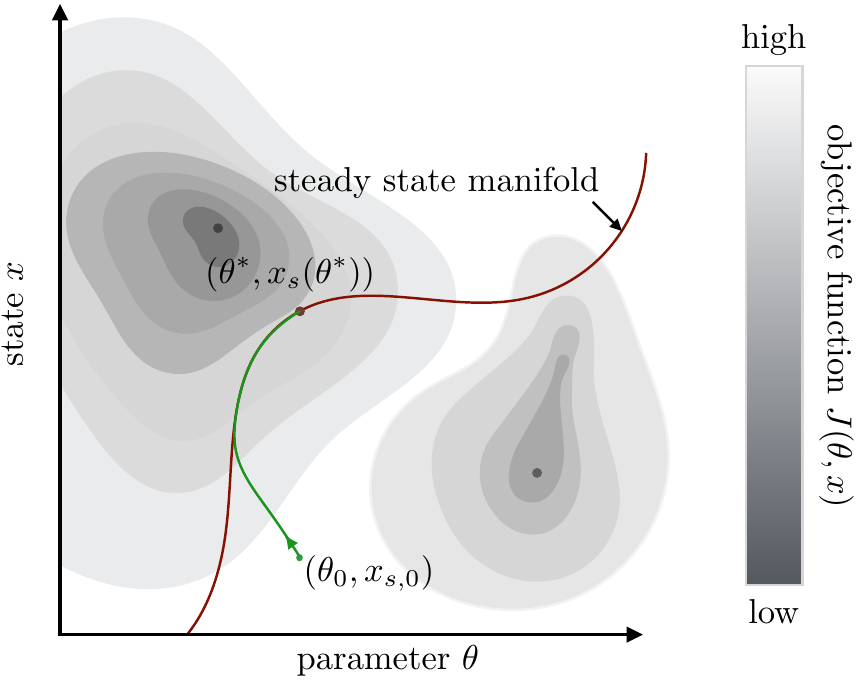}
\caption{Schematic illustration of steady state constrained optimisation problem. The path of a continuous analogue of a constrained optimiser ({\color{darkgreen}\bf---}) from a starting point $(\theta_0,x_{s,0})$ to a local minimum $(\theta^*,x_s(\theta^*))$ on the steady state manifold ({\color{darkred}\bf---}) is depicted.}
\label{fig: concept figure}
\end{centering}
\end{figure}

%%%%
%%%%
%%%%
\section{Geometry of the steady state manifold}
To develop a tailored method for solving optimisation problems with steady state constraints, we will exploit the first order geometry of the manifold of steady states. To this end we consider the sensitivities of the states $x$ with respect to the parameters $\theta$, 
\begin{equation*}
s_i := \frac{\partial x}{\partial \theta_i} = \left(\frac{\partial x_1}{\partial \theta_i}, \ldots, \frac{\partial x_{n_x}}{\partial \theta_i}\right)^T,
\end{equation*}
whose dynamics are governed by the forward sensitivity equation
\begin{equation*}
\dot{s}_i = \nabla_x f s_i + \nabla_{\theta_i} f, \quad s_i(0) =  \nabla_{\theta_i} x_0.
\end{equation*}
In matrix form, we obtain
\begin{equation}
\dot{S} = \nabla_x f S+ \nabla_{\theta} f, \quad S(0) =  \nabla_{\theta} x_0,
\label{eq: sens}
\end{equation}
with $S(t) = \begin{pmatrix} s_1(t),s_2(t),\ldots s_{n_\theta}(t)\end{pmatrix} \in \mathbb{R}^{n_x \times n_\theta}$.

In a steady state $x_s(\theta)$ corresponding to some parameter $\theta$, the forward sensitivity equation~\eqref{eq: sens} simplifies to
\begin{equation}
0 = \nabla_x f S + \nabla_{\theta} f
\end{equation}
evaluated at $(\theta,x_s(\theta))$. If the considered steady state $x_s(\theta)$ is locally exponentially stable, the Jacobian $\nabla_x f|_{(\theta,x_s(\theta))}$ is negative definite and invertible. Accordingly,
\begin{equation}
S = -(\nabla_x f)^{-1} \nabla_{\theta} f.
\end{equation}
Furthermore, this implies that there is a local parametrization of the steady state, $x_s(\theta)$ such that $(\theta,x_s(\theta))$ is an isolated root of \eqref{eq: ss}. The sensitivity of the steady state with respect to the parameters, $S$, can be used in the Taylor series of $x_s(\theta)$,
\begin{equation}
x_s(\theta + r \Delta \theta) = x_s(\theta) + S r \Delta \theta + O(r^2),
\label{eq: taylor with sens}
\end{equation}
with $S = S(\theta,x_s(\theta))$. The perturbation direction and the step size are denoted by $\Delta \theta$ and $r$, respectively. The Taylor expansion~\eqref{eq: taylor with sens} reveals how the steady state changes locally with the parameters. By reformulating~\eqref{eq: taylor with sens} and letting $r \rightarrow 0$, we obtain a dynamical system which evolves on $x_s(\theta)$,
\begin{equation}
\frac{dx_s}{dr} = S \Delta \theta.
\label{eq: ODE based on sens}
\end{equation}
Given an update direction $\Delta \theta$ and a length $r$,~\eqref{eq: ODE based on sens} provides the steady state $x_s(\theta + r \Delta \theta)$ up to the accuracy of the used ODE solver. Hence,~\eqref{eq: ODE based on sens} enables moves on the steady state manifold, similar to results in~\cite{AbsilMah2007}. A recalculation of the steady state for changed parameter values or a retraction operator are not longer required.

%%%%
%%%%
%%%%
\section{Optimisation with steady state constraints}

In this section, we exploit the insight into the local geometry of the steady state manifold to derive an ODE model converging to a locally optimal solution of the constrained optimisation problem~\eqref{eq: con opt prob}.

%%%%
\subsection{Simulation-based optimisation method.} \label{sec: method}
In local optimisation, the parameters $\theta$ are updated using the gradient of the objective function. The method of gradient descent, for instance, exploits an update
\begin{equation}
\theta_{k+1} = \theta_k - r \frac{dJ}{d\theta}^T,
\label{eq: update for theta}
\end{equation}
with step length $r$. The continuous analogue of gradient decent methods is $d\theta/dr = - (dJ/d\theta)^T$~\cite{Tanabe1985}. This ODE system can be coupled with the dynamical system evolving on the steady state manifold~\eqref{eq: ODE based on sens}, using $\Delta \theta = - (dJ/d\theta)^T$. We obtain the ODE system
\begin{equation}
\begin{aligned}
\frac{d\theta}{dr} &= - \left.\frac{dJ}{d\theta}\right|_{(\theta,x_s)}^T \\
\frac{dx_s}{dr} &= S \frac{d\theta}{dr} = - S \left.\frac{dJ}{d\theta}\right|_{(\theta,x_s)}^T,
\end{aligned}
\label{eq: ODE-based optimiser; Euler without retraction}
\end{equation}
with the steady state sensitivity $S$. Initialisation of~\eqref{eq: ODE-based optimiser; Euler without retraction} in a point $(\theta_0,x_{s,0})$ fulfilling~\eqref{eq: ss} yields a trajectory evolving on the steady state manifold, along which the objective function decreases.

The formulation~\eqref{eq: ODE-based optimiser; Euler without retraction} bears two disadvantages: (i)~An appropriate initial point $(\theta_0,x_{s,0})$ has to be determined by solving~\eqref{eq: ss}; and (ii)~numerical errors can result in a divergence from the steady state manifold. To address these problems, we introduce the term $\lambda f(\theta,x_s)$ which locally retracts the state of the system to the manifold by exploiting the stability properties of the steady state. This yields the system
\begin{equation}
\begin{aligned}
\frac{d\theta}{dr} &= - \left.\frac{dJ}{d\theta}\right|_{(\theta,x_s)}^T \\
\frac{dx_s}{dr} &= - \hat{S} \left.\frac{dJ}{d\theta}\right|_{(\theta,x_s)}^T + \lambda f(\theta,x_s).
\label{eq: ODE-based optimiser; Euler}
\end{aligned}
\end{equation}
For this modified system we do not require that the initial point $(\theta_0,x_{s,0})$ fulfils the steady state conditions~\eqref{eq: ss}, hence, the Jacobian $\nabla_x f|_{(\theta,x_s)}$ might not be invertible. To address this, we define
\begin{equation} 
\hat{S}(\theta,x_s) :=  -(\nabla_x f|_{(\theta,x_s)})^+ \nabla_{\theta} f|_{(\theta,x_s)},
\end{equation}
in which $(\nabla_x f|_{(\theta,x_s)})^+$ denotes the Moore--Penrose pseudoinverse of $\nabla_x f|_{(\theta,x_s)}$. On the steady state manifold, the Jacobian is invertible and we recover the standard steady state sensitivity. For large retraction factor $\lambda \gg 0$, the state $(\theta,x_s)$ is retracted quickly to the steady state manifold $(\theta,x_s(\theta))$.

\subsection{Analysis of simulation-based optimisation method.} \label{sec: stability}
In the following, we assess the local stability and convergence of~\eqref{eq: ODE-based optimiser; Euler} to a local optimum $(\theta^*,x_s(\theta^*))$ under four conditions.
\begin{assumption} \label{assumption: functional relationship}
There exists an locally isolated root $x_s(\theta)$ of the equation $f(\theta,x_s(\theta)) = 0$ for all $(\theta - \theta^*) \in B_{c} = \{\tilde \theta \in \mathbb{R}^{n_{\theta}}| \Vert \tilde \theta \Vert \leq c\} \subseteq \mathbb{R}^{n_\theta}$.
\end{assumption}
This root is a steady state of the system~\eqref{eq: model}.
\begin{assumption} \label{assumption: exponential stability}
The steady state $x_s(\theta)$ is locally exponentially stable uniformly in $\theta$, with $(\theta - \theta^*) \in B_{c}$.
\end{assumption}
Accordingly, $\exists \zeta_0, \gamma,k>0$ such that $\Vert x(t) -x_s(\theta)\Vert \leq k\Vert x(0)-x_s(\theta)\Vert \exp(-\gamma t)$, $\forall (x(0)-x_s(\theta)) \in B_{\zeta_0} \vcentcolon= \{\tilde \Delta_s(0)  \in \mathbb{R}^{n_x} |\Vert \tilde \Delta_s(0) \Vert < \zeta_0\}$ and $\forall (\theta-\theta^*) \in B_c$. This implies that $\exists \zeta \geq \zeta_0$ such that $(x(t)-x_s(\theta))\in B_{\zeta} \vcentcolon= \{\tilde \Delta_s(t)  \in \mathbb{R}^{n_x}|\Vert \tilde \Delta_s(t) \Vert < \zeta\}$~\cite{KhalilBook2002}.
\begin{assumption} \label{assumption: identifiability}
There exists a neighbourhood around $\theta^*$, $(\theta - \theta^*) \in B_{c}$, in which the objective function evaluated on the steady state manifold $x_s(\theta)$, $J(\theta,x_s(\theta))$, is locally strictly convex in $\theta$.
\end{assumption}
This implies that the parameters $\theta$ are locally structurally identifiable and that solutions to $d\theta/dr = -dJ/d\theta|_{(\theta,x_s(\theta))}$ converge to the optimal point $\theta^*$ for initial points $(\theta(0) - \theta^*) \in B_c$.
\begin{assumption} \label{assumption: smoothness}
The functions $(dJ/d\theta)|_{(\theta,x_s)}$, $f(\theta,x_s)$ and $x_s(\theta)$ and their partial derivatives up to order 2 are bounded for $(\theta - \theta^*) \in B_c$ and $(x_s - x_s(\theta)) \in B_\zeta$.
\end{assumption}
Using these assumptions we find:
 
\begin{theorem} \label{theorem: local convergence}
Let Assumptions~\ref{assumption: functional relationship} -  \ref{assumption: smoothness} be satisfied. Then there exists a $\lambda^*$ such that for all $\lambda > \lambda^*$ a local minimum $(\theta^*,x_s^*)$ of the optimisation problem~\eqref{eq: con opt prob} is a locally exponentially stable steady state of the system~\eqref{eq: ODE-based optimiser; Euler}.
\end{theorem}

\begin{proof}
To prove Theorem~\ref{theorem: local convergence} we use perturbation theory \cite{KhalilBook2002}. We define $\varepsilon = \lambda^{-1}$ and shift the optimum to the origin using the linear state transformation,
\begin{align*}
\tilde \theta = \theta - \theta^* \quad \text{and} \quad \tilde x_s = x_s - x_s(\theta^*).
\end{align*}
This yields the singular perturbed system
\begin{equation}
\begin{aligned}
\frac {d\tilde \theta}{dr} & = -\frac{dJ}{d\tilde \theta}^T =\vcentcolon \tilde F(\tilde \theta,\tilde x_s,\varepsilon)\\
\varepsilon \frac {d\tilde x_s}{dr} & = - \varepsilon \hat{S} \frac{dJ}{d\tilde \theta}^T+ f =\vcentcolon \tilde G(\tilde \theta,\tilde x_s,\varepsilon),
\end{aligned}
\label{eq: transformed opt system}
\end{equation}
with $dJ/d{\tilde \theta}$, $f$ and $\hat{S}$ evaluated at $(\tilde \theta+\theta^*,\tilde x_s+x_s(\theta^*))$. The system~\eqref{eq: transformed opt system} captures the dynamics of the deviances from the optimal parameter, $\tilde \theta$, and the deviances from the steady state for the optimal parameter, $\tilde x_s$. Furthermore, it possesses the following properties:
\begin{itemize}
\item[(i)] $\tilde F(0,0,\varepsilon) = 0$ and $\tilde G(0,0,\varepsilon) = 0$ as the objective function gradient vanishes, $0 = (dJ/d\tilde \theta)|_{(\theta^*,x_s(\theta^*))}^T$ and as the steady state condition is fulfilled, $0 = f(\theta^*,x_s(\theta^*))$. Both follows from optimality of $\theta^*$ and Assumption~\ref{assumption: functional relationship}.
\item[(ii)] The equation $0 = \tilde G(\tilde \theta,\tilde x_s,0)$ has the isolated root $\tilde x_s(\tilde \theta) = x_s(\tilde \theta + \theta^*) - x_s(\theta^*)$ (Assumption~\ref{assumption: functional relationship}) with $\tilde x_s(0) = 0$.
\item[(iii)] The functions $\tilde F$, $\tilde G$, and $x_s$ and their partial derivatives up to order~2 are bounded for $(\tilde{x}_s - x_s(\tilde \theta)) \in B_{\zeta}$ (Assumption~\ref{assumption: smoothness}).
\item[(iv)] The origin of the reduced systems 
\begin{equation*}
\frac {d\tilde \theta}{dr} = F(\tilde \theta,\tilde x_s(\tilde \theta),\varepsilon) = -\left.\frac{dJ}{d\tilde \theta}\right|_{(\tilde \theta + \theta^*,x_s(\tilde \theta) + x_s(\theta^*))}^T
\end{equation*}
obtained for $\varepsilon = 0$ is locally exponentially stable, with $x_s(\tilde \theta) + x_s(\theta^*) = x_s(\tilde \theta + \theta^*)$ from (ii). This follows from strict local convexity of the objective at $\theta^*$ (Assumption~\ref{assumption: identifiability}).
\item[(v)] The boundary-layer system is derived from~\eqref{eq: transformed opt system} using the transformation $\tilde \Delta_s = \tilde x_s - \tilde x_s(\tilde \theta) (= \tilde x_s - \tilde x_s(\tilde \theta))$, yielding
\begin{align*}
\frac {d\tilde \theta}{dr} & = -\frac{dJ}{d\tilde \theta}^T\\
\varepsilon \frac{d\tilde \Delta_s}{dr} & = \varepsilon (S - \hat{S}) \frac{dJ}{d\tilde \theta}^T+ f,
\end{align*}
with $dJ/d{\tilde \theta}$, $f$ and $\hat{S}$ evaluated at $(\tilde \theta+\theta^*,\tilde \Delta_s + x_s(\tilde \theta + \theta^*))$ and $S$ evaluated at $(\tilde \theta+\theta^*,x_s(\tilde \theta + \theta^*))$. After rescaling of the simulation time, $\rho = r/\varepsilon$, and setting $\varepsilon = 0$, we obtain the boundary-layer system
\begin{align*}
\frac{d\tilde \Delta_s}{d\rho} & = f(\tilde \theta+\theta^*,\tilde \Delta_s + x_s(\tilde \theta + \theta^*)).
\end{align*}
The origin of this boundary-layer system is locally exponentially stable, uniformly in $\tilde \theta$ as the steady state $x_s(\theta)$ of~\eqref{eq: model} is exponentially stable uniformly in $\theta = \tilde \theta+\theta^*$ (Assumption~\ref{assumption: exponential stability}).
\end{itemize}
The properties (i)-(v) are the prerequisites of~\cite[Theorem 9.3]{KhalilBook2002}, establishing the existence of $\varepsilon^* > 0$ such that for all $\varepsilon < \varepsilon^*$ systems of type~\eqref{eq: transformed opt system} are locally exponentially stable. As stability properties are conserved under the performed transformations, we obtain for $\varepsilon = \lambda^{-1}$ the Theorem~\ref{theorem: local convergence}. \hfill
\end{proof}

Theorem~\eqref{theorem: local convergence} establishes local exponential stability of local optima $(\theta^*,x_s(\theta^*))$ for appropriate choice of $\lambda$, assuming a convex objective function \eqref{eq: con opt prob} and a well-behaved system \eqref{eq: model}. From this follows also local convergence. Stability and convergence are not affected by the approximation of the steady state sensitivity via $\hat{S}$.

\subsection{Analysis of multiple steady-state constraints.}

In practice often steady state data for different inputs $u_i$ are available, yielding constraints $f(\theta,x_s^i,u_i) = 0$ for $i = 1, \ldots, m$. The simulation-based method and the stability proof presented in Sections~\ref{sec: method} and~\ref{sec: stability} can be generalised to this case. We simply concatenate constraints and state vectors, $\bar{f}(\theta,x) = (f(\theta,x_s^1,u_1),\ldots,f(\theta,x_s^m,u_m))$ with $x = (x_s^1,\ldots,x_s^m)$, and apply the method to the concatenated system. Due to the block structure some calculations, e.g., matrix inversion, can be simplified.
\\[3ex]
In the following, we illustrate the behaviour of the simulation-based optimisation method. Furthermore, the performance of the proposed approach will be compared to standard constrained and unconstrained optimisation methods. For this purpose, we consider two simulation examples for which the ground truth is known.

%%%%
%%%%
%%%%
\section{Example~1: Conversion reaction}

In this section, we study parameter estimation for conversion reactions from steady state data. Conversion reactions are among the most common motifs in biological systems, therefore particular interesting, and provide a simple test case.

%%%%
\subsection{Pathway model.}

We consider the conversion reaction 
\begin{equation*}
A \underset{\theta_2}{\overset{\theta_1}{\rightleftarrows}} B,
\end{equation*}
with parameters $\theta = \begin{pmatrix} \theta_1, \theta_2 \end{pmatrix} \in \mathbb{R}_+^2$. It is well-known that the concentration of the biochemical species $A$ is governed by
\begin{equation}
\begin{aligned}
\frac{dx}{dt} &= \theta_2 \xi - (\theta_1 + \theta_2) x, \quad x(0) = x_0
\end{aligned}
\label{eq: CR model}
\end{equation}
with $x_0 \in \mathbb{R}_+$ denoting the initial concentration of $A$ and $\xi \in \mathbb{R}_+$ denoting the sum of the initial concentrations of $A$ and $B$. The steady state of model~\eqref{eq: CR model} is
\begin{equation}
x_s(\theta) = \frac{\theta_2 \xi}{\theta_1 + \theta_2}.
\label{eq: CR steady state}
\end{equation}

%%%%
\subsection{Parameter estimation problem.}
To illustrate the properties of the simulation-based method, we consider a simple parameter estimation problem. We assume that for $\xi = 1$, the steady state $\bar{x}_s = 0.2$ is observed. Furthermore, for the parameters $\theta_i$, $i = 1,2$, prior knowledge suggests the value $\bar{\theta}_1 = 3.9$ and $\bar{\theta}_2 = 1.5$. Accordingly, we estimate the parameters using the weighted least squares optimisation problem
\begin{equation}
\begin{aligned}
\min_{\theta,x_s} &\, J(\theta,x_s) := \tfrac{1}{2}10 \left(x_s - \bar{x}_s\right)^2 + \tfrac{1}{2} \sum_{i=1}^2 \left(\theta_i - \bar{\theta}_{i}\right)^2\\
\mathrm{s.t.} \, &\, 0 = \theta_2 \xi - (\theta_1 + \theta_2)  x_s.
\end{aligned}
\label{eq: CR opt prob}
\end{equation}
%The analytical solution of this problem is given by 
%\begin{equation}
 %\theta_1 = \frac{\theta_2 (\chi-\bar{x}_s)}{\bar{x}_s}
%\end{equation}
 
%%%%
\subsection{Formulation of system for simulation-based optimisation.}
To formulate the ODE model solving~\eqref{eq: CR opt prob}, the gradient of the objective function with respect to the parameters
\begin{equation*}
\begin{aligned}
\frac{dJ}{d\theta_1} &= 10(x_s - \bar{x}_s) s_1 + (\theta_1 - \bar{\theta}_{1}) \\
\frac{dJ}{d\theta_2} &= 10(x_s - \bar{x}_s) s_2 + (\theta_2 - \bar{\theta}_{2})
\end{aligned}
\end{equation*}
and the local sensitivities 
\begin{equation*}
S(\theta,x_s) = \begin{pmatrix} s_1, s_2 \end{pmatrix} = \begin{pmatrix} \dfrac{- x_s}{\theta_1 + \theta_2}, \dfrac{\xi - x_s}{\theta_1 + \theta_2} \end{pmatrix} 
\end{equation*}
are derived. These components are substituted into~\eqref{eq: ODE-based optimiser; Euler}, yielding the system
\begin{equation}
\begin{aligned}
\frac{d\theta_1}{dr} &= 10(\bar{x}_s - x_s) s_1 + (\bar{\theta}_1 - \theta_1) \\
\frac{d\theta_2}{dr} &= 10(\bar{x}_s - x_s) s_2 + (\bar{\theta}_2 - \theta_2) \\
\frac{dx_s}{dr} &= \sum_{i=1}^2 10(\bar{x}_s - x_s) s_i^2 + \sum_{i=1}^2 (\bar{\theta}_i - \theta_i) s_i \\
&\hspace{13mm} + \lambda (\theta_2 \xi - (\theta_1^2 + \theta_2) x_s),
\end{aligned}
\label{eq: CR opt sys}
\end{equation}
with initial conditions $\theta_1(0) = \theta_{1,0}$, $\theta_2(0) = \theta_{2,0}$ and $x_s(0) = x_{s,0}$. It can be verified that the objective function $J$ is locally strictly convex in $\theta$ -- the parameters are locally identifiable -- and that the model~\eqref{eq: CR model} is asymptotically stable. Accordingly, system~\eqref{eq: CR opt sys} converges to a local optimum of the constrained optimisation problem~\eqref{eq: CR opt prob} (Theorem~\ref{theorem: local convergence}). As stopping criterion for the simulation-based optimisation we use $\max\{\Vert d\theta/dr \Vert ,\Vert dx_s/dr \Vert\}< TOL$ with $TOL = 10^{-6}$. Additionally, a maximal number of function evaluations is implemented.

%%%%
\subsection{Numerical results.}
To illustrate the simulation-based optimisation method we simulate system~\eqref{eq: CR opt sys} using the MATLAB ODE solver \texttt{ode15s}  with default settings. Exemplary trajectories are depicted in Figure~\ref{fig: CR Euler}. We find that for retraction factors $\lambda > 0$, the states $(\theta_1,\theta_2,x_s)^T$ converge to the optimal solution. As retraction renders the steady state manifold~\eqref{eq: CR steady state} attractive, also for initial conditions $(\theta_{1,0},\theta_{2,0},x_{s,0})^T$ which do not fulfil the steady state condition, fast convergence to the steady state manifold can be achieved using $\lambda \gg 1$. For large retraction ($\lambda \gg 1$), the dynamic consists of two phases: (Phase~1) the state $x$ converges quickly to the parameter-dependent steady state~\eqref{eq: CR steady state}; and (Phase~2) the state $(\theta_1,\theta_2,x_s)^T$ moves along the steady state manifold to a local optimum.

\begin{figure*}
\centering
\includegraphics[]{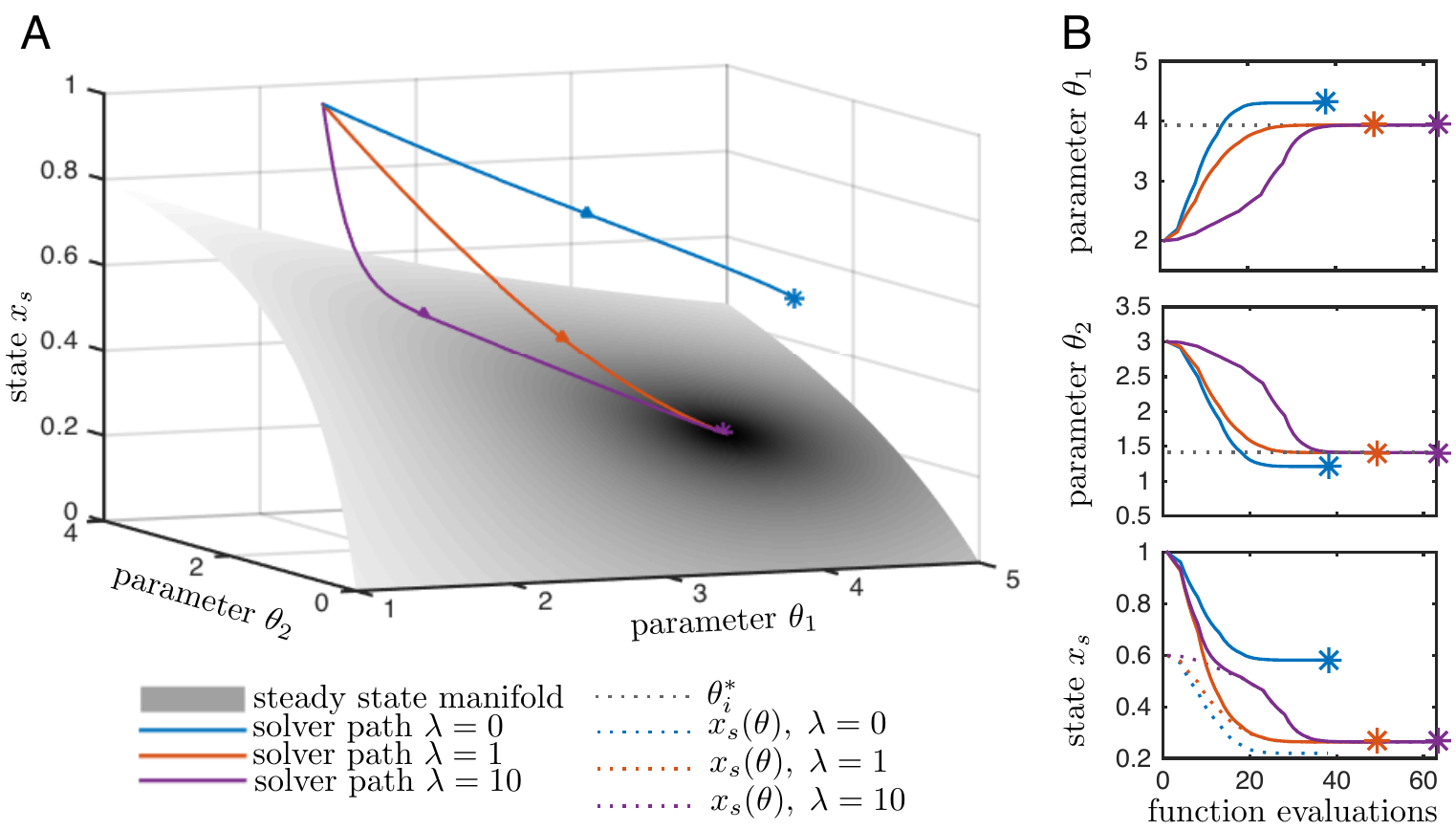}
\caption{Trajectory of the optimiser path for the estimation of the parameters of the conversion reaction model~\eqref{eq: CR model}. (A) The trajectory of the system~\eqref{eq: CR opt sys} for different retraction factors $\lambda$ (full lines), the endpoints (stars), and the steady state manifold (surface, \eqref{eq: CR steady state}) are shown. The objective function values are indicated by the surface coloring. (B) The trajectory of the individual components of system~\eqref{eq: CR opt sys} (full lines), their endpoints (stars), and optimal parameter value respective true steady state~\eqref{eq: CR steady state} for the parameters (dotted lines) are depicted.}
\label{fig: CR Euler}
\end{figure*}

%%%%
%%%%
%%%%
\section{Example~2: NGF-induced Erk activation}

To challenge the proposed simulation-based optimisation method, we consider in the following a larger and more realistic problem: Fitting of dose response data for NGF-induced Erk activation.

%%%%
\subsection{Pathway model.}
NGF is known to induce a strong pain sensitisation during inflammation via the activation of Erk. As this is clinically relevant, a simple model for NGF-induced Erk activation has been developed in~\cite{HasenauerHas2014},
\begin{equation}
\begin{aligned}
\frac{dx_1}{dt} &= 10^{\theta_1} u \left(10^{\theta_5} - x_1\right) - 10^{\theta_2} x_1 \\
\frac{dx_2}{dt} &= \left(x_1 + 10^{\theta_3}\right) \left(10^{\theta_6} - x_2\right) - 10^{\theta_4} x_2
\end{aligned}
\label{eq: NE model}
\end{equation}
in which $x_1$ denotes the activity of the NGF receptor TrkA, $x_2$ denotes the activity of Erk, and $u$ denotes the concentration of NGF. The parameters $\theta \in \mathbb{R}^6$ are logarithms of effective rate parameters. The steady state of model~\eqref{eq: NE model} is
\begin{equation}
\begin{aligned}
x_{s,1}(\theta) &= \frac{10^{\theta_1+\theta_5} u}{10^{\theta_1} u + 10^{\theta_2}}\\
x_{s,2}(\theta) &= \frac{10^{\theta_6} (x_{s,1}(\theta) + 10^{\theta_3})}{x_{s,1}(\theta) + 10^{\theta_3} + 10^{\theta_4}}
\end{aligned}
\label{eq:NEsteadystate}
\end{equation}
For details on the model, we refer to the original publication~\cite{HasenauerHas2014}.

%%%%
\subsection{Parameter estimation problem.}
The parameters $\theta$ are unknown and have to be estimated from experimental data. In practice, the activity of the NGF receptor TrkA ($x_1$) cannot be assessed and merely data for the Erk activity level ($x_2$) are available. These data can be time resolved, but mostly dose response curves for the steady state are recorded. 

In the following, we assume that for ten different NGF concentrations ($u_i$, $i = 1,\ldots,10$) the stationary Erk activity ($\bar{x}_{s,2}^i$, $i = 1,\ldots,10$) has been measured. Therefore, we evaluate the steady state~\eqref{eq:NEsteadystate} for these different inputs $u_{i}$. Using a standard least squares objective function, we obtain the constrained optimisation problem
\begin{equation}
\begin{aligned}
\min_{\theta,\{x_s^i\}_i} &\, J(\theta,\{x_s^i\}_i) = \tfrac{1}{2} \sum_{i=1}^{10} \left(x_{s,2}^i - \bar{x}_{s,2}^i\right)^2 \\
\mathrm{s.t.} \hspace*{1mm} 
&\, 0 = 10^{\theta_1} u_i \left(10^{\theta_5} - x_{s,1}^i\right) - 10^{\theta_2} x_{s,1}^i \\
&\, 0 = \left(x_{s,1}^i + 10^{\theta_3}\right) \left(10^{\theta_6} - x_{s,2}^i\right) - 10^{\theta_4} x_{s,2}^i \\[1ex]
& \hspace*{30mm} \quad i = 1,\ldots,10.
\end{aligned}
\label{eq: NE opt prob}
\end{equation}
This constrained optimisation problem has $2\cdot10 + 6 = 26$ optimisation variables. As an analytical expression for the steady state~\eqref{eq:NEsteadystate} is available the constrained optimisation problem~\eqref{eq: NE opt prob} can be reduced to an unconstrained optimisation problem by substituting the expression for the parameter dependent steady state~\eqref{eq: NE opt prob} into the objective function $J(\theta,x_s)$, yielding
\begin{equation}
\begin{aligned}
\min_{\theta} &\, \tfrac{1}{2} \sum_{i=1}^{10} \left(\frac{10^{\theta_6} (\frac{10^{\theta_1+\theta_5} u_i}{10^{\theta_1} u_i + 10^{\theta_2}} + 10^{\theta_3})}{\frac{10^{\theta_1+\theta_5} u_i}{10^{\theta_1} u_i + 10^{\theta_2}} + 10^{\theta_3} + 10^{\theta_4}} - \bar{x}_{s,2}^i\right)^2.
\end{aligned}
\label{eq: NE opt prob red}
\end{equation}
This unconstrained optimisation problem has $6$ optimisation variables, and the size is therefore independent of the number of NGF concentrations.

%%%%
\subsection{Formulation of system for simulation-based optimisation.}
As before, a dynamical system converging to the locally optimal points of the constrained optimisation problem~\eqref{eq: NE opt prob} can be formulated. To this end the objective function gradient, $dJ/d\theta$, and the local sensitivities, $S(x,\theta)$, are derived. This can be done manually, but for high-dimensional problems symbolic math toolboxes should be used. The results are inserted in~\eqref{eq: ODE-based optimiser; Euler} to yield a problem specific system.

%%%%
\subsection{Numerical results.}
In the following, we compare our simulation-based optimisation method with two state-of-the-art implementations of standard methods:
\begin{itemize}
\item Constrained optimisation~\eqref{eq: NE opt prob} using the \mbox{MATLAB} algorithm \texttt{fmincon} for constrained nonlinear optimisation; and
\item Unconstrained optimisation~\eqref{eq: NE opt prob red} using the \mbox{MATLAB} algorithm \texttt{fminunc} for unconstrained nonlinear optimisation.
\end{itemize}
Constrained optimisation using \texttt{fmincon} (or a similar algorithm) is the n\"aive approach. This approach does neither require an analytical formula for the steady state nor does it exploit the structure of the problem. Unconstrained optimisation~\eqref{eq: NE opt prob red} on the other hand requires an analytical solution of the steady state. Using this analytical solution the problem is simplified significantly. Accordingly, we expect that unconstrained optimisation will provide a very good performance, its applicability is however limited to simple models for which an analytical expression for the steady state is known. For the evaluation, \texttt{fmincon} is supplied with the objective function value and the values of the constraint, as well as the respective analytical derivatives. For \texttt{fminunc} only the objective function and its gradient are implemented. Both algorithms use default settings. The ODE system obtained for the simulation-based method is simulated using the MATLAB solver \texttt{ode15s} with default settings. Artificial data for the comparison are generated by calculating the analytical steady state~\eqref{eq:NEsteadystate} for inputs $u=(0,0.01,0.05,0.1,0.5,1,5,10,50,100)$ and adding normally distributed noise with zero mean and variance of 0.01.

We used constrained optimisation, unconstrained optimisation as well as the proposed simulation-based method to estimate the parameters. For all methods the initial parameter estimates were sampled from a uniform distribution, $\theta_k \sim \mathcal{U}(-3,1)$. In addition to initial parameters, constrained optimisation and simulation-based methods require initial estimates for the steady states. They are also sampled from a uniform distribution, $x_{s,j}^i \sim \mathcal{U}(0,3)$ for $j = 1,2$ and $i = 1,\ldots,10$.% Representative examples for runs of the simulation-based methods are depicted in Figure~\ref{fig: NE ex run}. We find that the state converges reliably to the steady state manifold (Figure~\ref{fig: NE ex run}, bottom), thereby (potentially) increasing the objective function value (Figure~\ref{fig: NE ex run}, top), and then moves on the steady state manifold towards the next local minimiser. This demonstrates that the proposed methods also works for higher-dimensional problems and that it shows the expected behaviour.

%\begin{figure}[t]
%\includegraphics[scale=1]{ErkDR.pdf}
%\caption{Representative trajectories of gradient descent-type simulation-based optimisation methods for NGF-induced Erk activation. Objective function (A), parameter values (B) and steady state for different NGF concentrations $u_i$ (C) are depicted for the simulation-based method with retraction factors $\lambda = 3$ (left) and $\lambda = 20$ (right).}
%\label{fig: NE ex run}
%\end{figure}

To assess the convergence properties and the computation time, all aforementioned methods were initialised with 100 sampled starting point. The final objective function values as well as the respective computation time is depicted in Figure~\ref{fig: NE comp obj}. We find -- as expected -- that the unconstrained optimisation method is on average computationally most efficient and for almost all runs a same final objective function value is achieved (which is probably the global optimum). The simulation-based optimisation methods with retraction factor $\lambda = 2$ and $\lambda = 20$ possess similar convergences properties (Figure~\ref{fig: NE comp obj}A), the computation time is however roughly 10-times higher. The standard constrained optimisation method implemented by \texttt{fmincon} possesses a significantly worse convergence rate of roughly 28\%. The computation time for \texttt{fmincon} varies over several orders of maginitude and the average is comparable to the computation time for the simulation-based optimisation methods (Figure~\ref{fig: NE comp obj}B). However, the mean computation time per converged start is roughly 7 to 8 times higher for \texttt{fmincon} in comparison to the simulation based method with retraction factor $\lambda = 20$ (\texttt{fmincon}: 7.56 seconds per converged start; simulation based with $\lambda = 20$: 0.98 seconds per converged start).
%For lower retraction factors not all runs reach the optimal objective function value and the computation time is increased further. Nevertheless, all simulation-based methods outperform for this example the standard constrained optimisation methods. The reason is the high-dimensional nonlinear constraint. A detailed analysis revealed that in many cases no solution to the steady state is found, and even if the method is starting on the steady state manifold, updates are inefficient. This also explains the high computation time. 
%An increase of the number of steady state constraints (here related to the number of different NGF concentrations) causes a further performance drop for the constrained optimisation approach, while unconstrained optimisation and simulation-based methods can deal with it rather well.

In general it is difficult to ensure that the objective function is locally strictly convex -- a prerequisite for Theorem~\ref{theorem: local convergence} -- prior to the optimisation. Also for this example, this was not possible, and indeed it is also not the case. The optimisation results reveal that the parameters are practically non-identifiable and that there is a continuum of parameters which yields the same optimal objective function value. As the simulation-based methods worked fine despite the lack of identifiability, it is well suited for practical applications and we expect that Theorem~\ref{theorem: local convergence} can be extended.

\begin{figure}[t]
\begin{centering}
\includegraphics[scale=1]{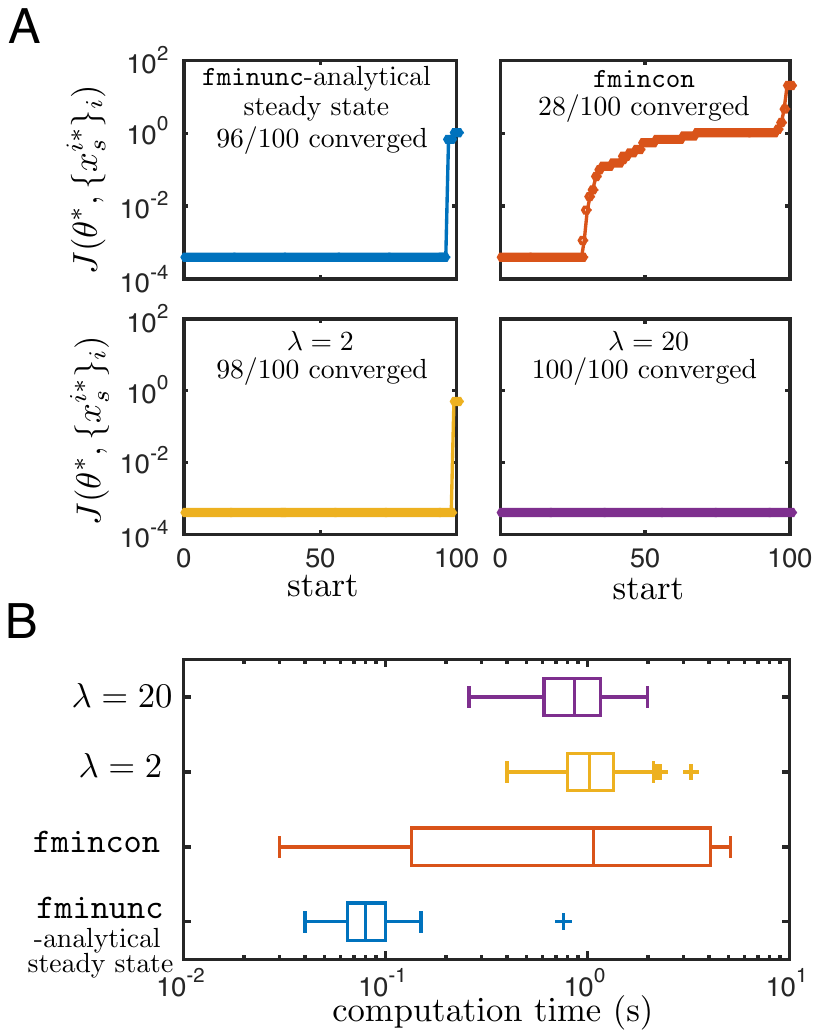}
\caption{Comparison optimisation methods for NGF-induced Erk activation model. (A) Final objective function values and (B) computation time for 100 runs of unconstrained optimisation method (\texttt{fminunc}), constrained optimisation algorithm (\texttt{fmincon}) and the proposed simulation-based optimisation algorithm with $\lambda = 2$ and $\lambda = 20$ are depicted.}
\label{fig: NE comp obj}
\end{centering}
\end{figure}

In summary, for this estimation problem, the simulation-based method outperforms standard constrained optimisation methods and achieves a performance only sightly worse than unconstrained optimisation methods. Furthermore, simulations-based methods seem to be robust with respect to non-identifiability. While we use artificial data to evaluate our method, similar experiment data are available (see~\cite{HasenauerHas2014}). This renders also this application example important for practice.

%%%%
%%%%
%%%%
\section{Discussion and Outlook}
Optimisation problems with steady state constraints are often challenging. If an analytical expression for the parameter and input dependent steady state, $x_s(\theta,u)$ is not available, the vector of optimisation variables contains the unknown parameters as well as the corresponding steady states. Accordingly, optimisers have to evolve on the non-linear manifold defining the set of steady states. In this manuscript, we propose a continuous analogue of a gradient descent algorithm for manifold optimisation. This simulation-based method exploits the local geometry of the steady state manifold for optimisation. The local asymptotic stability of the steady state is used to render the steady state manifold attractive. We establish a result for local convergence using perturbation theory.

The proposed simulation-based optimisation method is evaluated using two models for biological processes. We find that the simulation-based optimisation method possesses improved convergence properties in comparison to standard constrained optimisation methods implemented in the MATLAB routine \texttt{fmincon}. Furthermore, the simulation-based optimisation method yields convergence properties almost comparable to those of unconstrained optimisation methods exploiting an analytical expression for the steady state. The proposed simulation-based optimisation method is however applicable to a broader class of problems as an analytical expression for the steady state is not necessary.

An open question is how the proposed method behaves in the presence of practical and structural non-identifiability. Preliminary results suggest that the method always provides a point on the non-identifiable subspace but this has to be studied in more detail. As in traditional optimisation, also using alternative update schemes, e.g. Newton-type instead of gradient descent-based parameter updates, might yield even more convincing convergence results. As shown in the proof, the retraction factor $\lambda$ has to be chosen large to ensure convergence. Howerver, a large $\lambda$ will render the simulation-based system stiff. An intelligent choice of $\lambda$ is therefore necessary. Furthermore, extension towards a combination of steady state and time series data might be interesting. The simulations-based approach can in principle be generalised to other model classes, including partial differential equations.

%%%%
%%%%
%%%%

%\bibliography{/Users/janhasenauer/Documents/02_work/08_literature/Database}
%\bibliography{Database}

\end{document}